\documentclass[preprint, 11pt]{elsarticle}
\usepackage[margin=1.5in]{geometry}

\usepackage{graphicx}
\usepackage{tikz}
\usetikzlibrary{positioning,arrows}
\usepackage{etoolbox}
\usepackage{setspace} 
\usepackage{tikz-cd}
\usepackage[nointegrals]{wasysym}
\usepackage{adjustbox}
\usepackage{mathrsfs}
\usepackage{mathtools}
\usepackage[colorinlistoftodos]{todonotes}
\usepackage[labelfont=bf]{caption}

\newcommand{\abs}[1]{\lvert #1 \rvert}

\newcommand{\res}[2]{{\rm res}_{#1}^{#2}}
\newcommand{\tr}[2]{{\rm tr}_{#1}^{#2}}
\newcommand{\N}[2]{{\rm N}_{#1}^{#2}}
\newcommand{\burn}{\underline{A}}
\newcommand{\burnghost}{\underline{\tilde{A}}}

\newcommand{\NN}{\mathbb{N}}

\newcommand{\ZZ}{\mathbb{Z}}

\newcommand{\Spec}{\operatorname{Spec}}

\newcommand{\tamb}[1]{\underline{#1}}
\newcommand{\sr}[1]{\mathscr{#1}}
\newcommand{\fr}[1]{\mathfrak{#1}}
\newcommand{\gent}[1]{(\!(#1)\!)}

\usepackage{amssymb}
\usepackage{amsmath}
\usepackage{amsthm}
\usepackage{amsfonts}
\usepackage{amsxtra}
\usepackage[all,2cell]{xy}
\usepackage{verbatim}
\usepackage{color}
\usepackage{xcolor}
\usepackage[unicode=true, %pdfusetitle,
bookmarks=true,bookmarksnumbered=false,
breaklinks=false,
backref=false,
colorlinks=true,
linkcolor=blue,
citecolor=blue,
urlcolor=blue,
final
]{hyperref}

\usepackage[capitalise]{cleveref}

\newcommand{\newrefformat}[2]{}

\usepackage{tikz}

\theoremstyle{plain}
\newtheorem{theorem}{Theorem}[section]
\makeatletter\let\c@theorem\c@theorem\makeatother
\newtheorem{lemma}{Lemma}[section]
\makeatletter\let\c@lemma\c@theorem\makeatother
\newtheorem{corollary}{Corollary}[section]
\makeatletter\let\c@corollary\c@theorem\makeatother
\newtheorem{proposition}{Proposition}[section]
\makeatletter\let\c@proposition\c@theorem\makeatother

\makeatletter\let\c@conjecture\c@theorem\makeatother
\newtheorem*{theorem*}{Theorem}

\theoremstyle{definition}
\newtheorem{definition}{Definition}[section]
\makeatletter\let\c@definition\c@theorem\makeatother

\makeatletter\let\c@example\c@theorem\makeatother

\makeatletter\let\c@xca\c@theorem\makeatother

\theoremstyle{remark}
\newtheorem{remark}{Remark}[section]
\makeatletter\let\c@remark\c@theorem\makeatother

\makeatletter
\let\c@equation\c@theorem
\numberwithin{equation}{section}
\makeatother

\theoremstyle{plain}
\newcommand{\thistheoremname}{}
\newtheorem*{genericthm*}{\thistheoremname}

\crefname{lemma}{Lemma}{Lemmas}
\crefname{thm}{Theorem}{Theorems}
\crefname{defn}{Definition}{Definitions}
\crefname{notn}{Notation}{Notations}
\crefname{const}{Construction}{Constructions}
\crefname{prop}{Proposition}{Propositions}
\crefname{rem}{Remark}{Remarks}
\crefname{cor}{Corollary}{Corollaries}
\crefname{equation}{Equation}{Diagrams}
\crefname{ex}{Example}{Examples}

\begin{document}

\begin{frontmatter}

\title{The Spectrum of the Burnside Tambara Functor of a Cyclic Group}

\author{Maxine Calle\corref{cor1}\fnref{fn1}}
\ead{callem@sas.upenn.edu}
\cortext[cor1]{Corresponding author}
\fntext[fn1]{Present address: University of Pennsylvania, PA}

\author{Sam Ginnett}
\ead{ginnetts@reed.edu}

\address[1]{Reed College, OR}
\address[2]{Reed College, OR}

\begin{abstract}
We derive a family of prime ideals of the Burnside Tambara functor for a finite group $G$. In the case of cyclic groups, this family comprises the entire prime spectrum. We include some partial results towards the same result for a larger class of groups.
\end{abstract}
\begin{keyword}
Tambara functor \sep Burnside ring \sep cyclic group \sep prime spectrum \\ \MSC[2020]{\textbf{19A22}, 55P91, 18B99, 13A15}
\end{keyword}
\end{frontmatter}
%--------------------------------------------------------------%--------------------------------------------------------------

\section{Introduction}

Our work determines the spectrum of the Burnside Tambara functor on cyclic groups, expanding on the related works of Nakaoka  \cite{nakaoka:2011}, Lewis \cite{lewis:1980}, and Dress \cite{dress:1971}. In \cite{nakaoka:2011}, Nakaoka develops an analogy between the theory of the ideals of Tambara functors, defined as kernels of Tambara functor morphisms, and the theory of ideals of commutative rings. Using the appropriate definitions, many of the properties of ideals of commutative rings have direct analogues in the theory of Tambara functors. In particular, Tambara functors have a spectrum, consisting of all prime ideals. The Tambara spectrum is functorial in Tambara morphisms just as the Zariski spectrum is functorial in ring homomorphisms.

Nakaoka calculates the prime spectrum of the Burnside Tambara functor for cyclic $p$-groups in  \cite{nakaoka:2012, nakaoka:2014}. The Burnside Tambara functor $\burn_G$, formed by the collection of Burnside rings for each subgroup of a given finite group $G$, is the initial object in the category of Tambara functors over $G$ and therefore plays a role analogous to $\ZZ$ in the theory of commutative rings. In \cite{dress:1971}, Dress derives the spectrum of the Burnside ring, showing that the prime ideals can be indexed by a subgroup and a number $p$ that is either prime or zero. This is generalized by Lewis \cite[\S 6]{lewis:1980} to determine the prime ideals of the Burnside Green functor. 
%To date, the only calculations of spectra of Tambara functors appearing in the literature are in the papers \cite{nakaoka:2012} \todo{I Honestly cant find the original publication}\todo{Was it this one?} and \cite{nakaoka:2014}.

In this paper, we first identify a family of prime ideals of the Burnside Tambara functor for any finite group $G$. This family is indexed by a subgroup $C\leq G$ and a number $p$ that is either prime or zero, similar to the result of Dress on the spectrum of the Burnside ring and Lewis on the spectrum of the Burnside Green functor. In the specific case of cyclic groups, we are able to determine that this family comprises the entire spectrum. We suspect that this result holds more generally, and offer some suggestions for future work along these lines. Our main results can be summarized as follows:

\begin{theorem}[\cref{thm:PCp prime,thm:main_thm,thm:containment}]
Let $G$ be a finite group. Then for any $C\leq G$ and $p$ a prime or zero,
$\fr{p}_{C, p}$ is a prime ideal of $\burn_G$, where $\fr{p}_{C, p}$ is defined in \cref{def:the_ideals}.

If $G$ is Abelian, then we can establish some containment rules. For any $H,K\leq G$ and $p,q$ prime, we have\begin{enumerate}
    \item[(i)] $\fr{p}_{H, 0}\subseteq \fr{p}_{K, 0}$ if and only if $K \leq H$,
    \item[(ii)] $\fr{p}_{H, 0}\subset\fr{p}_{H, p}$ and $\fr{p}_{H, p}\not\subseteq \fr{p}_{K, 0}$,
    \item[(iii)] $\fr{p}_{H, p}\subseteq \fr{p}_{K, q}$ if and only if $p = q$ and $O^p(K)\leq O^p(H)$,
\end{enumerate} where $O^p(H)$ is defined in \cref{thm:dress containment}.
Furthermore, if $G$ is cyclic, then $$\Spec(\burn_G) = \{\fr{p}_{C, p} \mid C \leq G,~ p \text{ prime or zero}\}.$$ 
\end{theorem}

For a cyclic group, the containment structure of the ideals that ``lie over'' $p$ is identical to the subgroup structure of $C_{n/p^\nu}$, where $p^\nu$ is the highest power of $p$ that divides $n$. (In the case of $p=0$, we have the subgroup lattice structure of $C_n$.) These relations are illustrated for $n=12$ in \cref{fig:incl C12} below, and a more general picture is given in \cref{fig:incl Cn}. Note that the containment structure of the ideals $\fr{p}_{C_i, p}$ for $p=0$ or $p\nmid n$ is dual to the subgroup containment structure of $C_n$. 

\begin{figure}[h!]
    \centering
\begin{tikzpicture}[node distance=1.5cm,line width=0.25pt]
\title{Title}
\node(0) at (0,0)  {$\fr{p}_{e,0}$};
\node(C3)       [above left  =0.5cm of 0] {$\fr{p}_{C_3,0}$};
\node(C2)       [above right  =0.5cm of 0] {$\fr{p}_{C_2,0}$};
\node(C6)       [above =1cm of C3] {$\fr{p}_{C_{6},0}$};
\node(C4)       [above =1cm of C2] {$\fr{p}_{C_4,0}$};
\node(C12)       [above =3cm of 0] {$\fr{p}_{C_{12},0}$};

\draw[<-] (0) -- (C3);
\draw[<-] (0) -- (C2);
\draw[<-] (C2) -- (C4);
\draw[<-] (C3) -- (C6);
\draw[<-] (C2) -- (C6);
\draw[<-] (C6) -- (C12);
\draw[<-] (C4) -- (C12);

\node(p) at (4.5,0.5)  {$\fr{p}_{e,p}$};
\node(C3p)       [above left  =0.5cm of p] {$\fr{p}_{C_3,p}$};
\node(C2p)       [above right  =0.5cm of p] {$\fr{p}_{C_2,p}$};
\node(C6p)       [above =1cm of C3p] {$\fr{p}_{C_{6},p}$};
\node(C4p)       [above =1cm of C2p] {$\fr{p}_{C_4,p}$};
\node(C12p)       [above =3cm of p] {$\fr{p}_{C_{12},p}$};

\draw[gray, ->] (0) -- (p);
\draw[gray, ->] (C2) -- (C2p);
\draw[gray, ->] (C3) -- (C3p);
\draw[gray, ->] (C4) -- (C4p);
\draw[gray, ->] (C6) -- (C6p);
\draw[gray, ->] (C12) -- (C12p);

%otherwise
\draw[<-] (p) to[bend left=10] (C3p);
\draw[<-] (p) to[bend right=10] (C2p);
\draw[<-] (C2p) to[bend right=10] (C4p);
\draw[<-] (C3p) to[bend left=10] (C6p);
\draw[<-] (C2p) to[bend left=10] (C6p);
\draw[<-] (C6p) to[bend left=10] (C12p);
\draw[<-] (C4p) to[bend right=10] (C12p);

%p=2
\draw[red, ->] (p) to[bend left=10] (C2p);
\draw[red, ->] (C2p) to[bend left=10] (C4p);

\draw[red, ->] (C3p) to[bend right=10] (C6p);
\draw[red, ->] (C6p) to[bend right=10] (C12p);

\node(L2)  at (8,3) {\textcolor{red}{$p=2$}};

%p=3
\draw[blue, ->] (p) to[bend right=10] (C3p);
\draw[blue, <-] (C6p) to[bend left=10] (C2p);
\draw[blue, ->] (C4p) to[bend left=10] (C12p);

\node(L3)  [below=0.1cm of L2] {\textcolor{blue}{$p=3$}};

\end{tikzpicture}
    \caption{\textbf{Inclusion structure of the prime ideals of $\burn_{C_{12}}$.} Every ideal in $\Spec(\burn_{C_{12}})$ is of the form $\fr{p}_{C_i, p}$ for some $i\mid 12$ and $p$ prime or $0$. There is an arrow from $\fr{p}_{C_i,p}$ to $\fr{p}_{C_j, q}$ if $\fr{p}_{C_i,p}\subseteq\fr{p}_{C_j, q}$. In the case where $p=2,3$, some of the inclusions become equalities, as indicated by the red and blue arrows, respectively.}
    \label{fig:incl C12}
\end{figure}
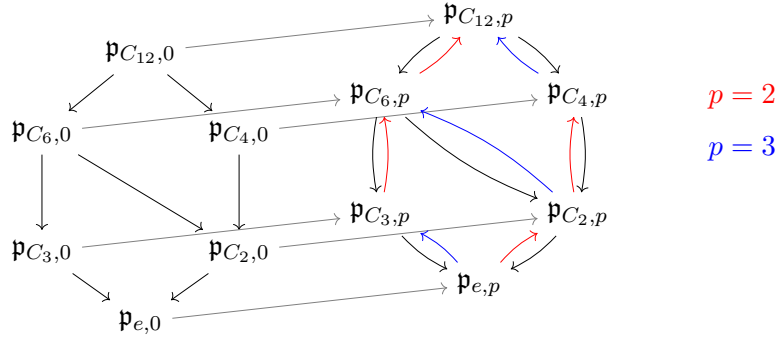

Our results for Abelian $G$ establish a bijection \[
\Spec(A(G))\longleftrightarrow \{\fr{p}_{C,p}\mid C\leq G, ~p\text{ prime or }0\}\subseteq \Spec(\burn_{G}),
\] however, this bijection is not continuous (see \cref{prop:no homeom}). By comparing the inclusion structures, we can clearly see that the Krull dimension of the two spectra differs.

\subsection*{Outline.} 
The second section recalls necessary background on Tambara functors and their ideals, with particular focus on the Tambara structure of the Burnside ring. 
Inspired by the work of A.~Dress \cite{dress:1971}, we establish a family of prime ideals of the Burnside functor in \cref{sec:collection of prime ideals} and examine their inclusion structure.
The remainder of the paper is dedicated to showing this family comprises the entire spectrum for the case of a cyclic group: we first describe the ideals in terms of generators in \cref{sec:gens} and then we complete the proof in \cref{sec:pf}. 

\subsection*{Acknowledgements.} The authors would like to thank both Kyle Ormsby and Ang\'elica Osorno, as well as the rest of the Reed College Collaborative Mathematics Research Group (CMRG), for their guidance and support. We are also very grateful to the referee for their many helpful and generous comments, particularly their simplified proof of \cref{thm:PCp prime}.

This work began when the authors were undergraduate participants in the CMRG during Summer 2019, under the guidance of K. Ormsby and  A. Osorno, with generous funding from NSF Grant DMS-1709302. The work in the final few sections, also begun during the summer program, was completed under the supervision of K. Ormsby the following academic year. The first author is now supported by an NSF Graduate Research Fellowship (DGE-1845298).

%--------------------------------------------------------------%--------------------------------------------------------------
%--------------------------------------------------------------%--------------------------------------------------------------

\section{Background: Tambara Ideals and the Burnside Spectrum}

\subsection{Tambara functors and prime ideals.} 
Tambara functors \cite{tambara:1993} are elaborations of Mackey functors with multiplicative norm maps in addition to restrictions and transfers. We will consider Tambara functors on a finite group $G$, although we note that the definition of Tambara functor may be extended to profinite groups \cite{nakaoka:2009}. 

\begin{definition}\label{defn:tamb functor}
Let $G$\textbf{Fin} denote the category of finite $G$-sets for a finite group $G$. A \textit{Tambara functor} $T$ on $G$ is a triple $(T^*, T_+, T_{\boldsymbol{\cdot}})$ where $T^*$ is a contravariant functor $G$\textbf{Fin}$\rightarrow$\textbf{Set} and $T_+, T_{\boldsymbol{\cdot}}$ are covariant functors $G$\textbf{Fin}$\rightarrow$\textbf{Set} such that 
\begin{enumerate}
    \item $(T^*, T_+)$ is a Mackey functor on $G$,
    \item $(T^*, T_{\boldsymbol{\cdot}})$ is a semi-Mackey functor on $G$, and
    \item given an exponential diagram
    \begin{center}
    \begin{tikzcd}
    X \arrow{d}[swap]{f} & \arrow{l}[swap]{p} A & \arrow{l}[swap]{\lambda} Z \arrow{d}{\rho}\\
    Y  & & B \ar[ll, "q"]
    \end{tikzcd}
    \end{center}
    in $G$\textbf{Fin} (in the sense of \cite{tambara:1993}), the diagram 
    \begin{center}
    \begin{tikzcd}
    T(X) \arrow{d}[swap]{T_\cdot(f)} & \arrow{l}[swap]{T_+(p)} T(A) \arrow{r}{T^*(\lambda)} &  T(Z) \arrow{d}{T_\cdot(\rho)}\\
    T(Y) & & T(B) \arrow{ll}{T_+(q)}
    \end{tikzcd}
    \end{center}
    commutes.
\end{enumerate}
For the sake of brevity, we use the notation $f_{+}:=T_{+}(f)$, $f_{\boldsymbol{\cdot}}:=T_{\boldsymbol{\cdot}}(f)$, and $f^*:=T^*(f)$.
\end{definition}

As is the case for Mackey functors, it suffices to specify how a Tambara functor behaves on transitive $G$-sets.  This observation prompts a second characterization of Tambara functors in terms of subgroups of $G$:

\begin{definition}
A Tambara functor $T$ on a finite group $G$ is completely specified by a ring $T(G/H)$ for all $H\leq G$ and the following maps for all subgroups $L\leq H\leq G$:
\begin{enumerate}
    \item Restriction $\res{L}{H} := q^*$
    \item Transfer $\tr{L}{H} := q_+$
    \item Norm $\N{L}{H} := q_{\boldsymbol{\cdot}}$
    \item Conjugation $c_{g, H} := (c_g)^*$
\end{enumerate}
where $q\colon G/K\rightarrow G/H$ is the quotient map and $c_g\colon G/H^g \rightarrow G/H$ is conjugation-by-$g$ with $H^g=g^{-1}Hg$. These maps must satisfy a number of compatibility conditions as specified M.~Hill and K.~Mazur in \cite{mazur:2019}. 
\end{definition}

Notably, while restriction and conjugation are ring maps, transfer and norm only respect the additive and multiplicative structures, respectively. Tambara functors form a category, the \textit{Tambara functor category}, and the morphisms in this category are natural transformations between Tambara functors.

\begin{definition}
A \textit{Tambara functor morphism} $\varphi\colon T\to S$ is a collection of ring homomorphisms $\varphi_X\colon T(X)\rightarrow S(X)$ for all finite $G$-sets $X$ which form a natural transformation of each of the three component functors.
\end{definition}

There are many parallels between Tambara functor theory and ring theory; for example, the \textit{kernel} of $\varphi$ is given by the collection of kernels of the associated ring maps. As in standard ring theory, the kernel of $\varphi$ will be an ideal of the domain, but now understood as a Tambara ideal. This work in particular will focus on prime Tambara ideals, which we introduce presently.

\begin{definition}\label{defn:tamb ideal}
An \textit{ideal} $\sr{I}$ of $T$ consists of a collection of standard ring-theoretic ideals $\sr{I}(G/H) \subseteq T(G/H)$ for each open $H\leq G$, such that for all open subgroups $K\leq H$
\begin{enumerate}
    \item $\res{K}{H}(\sr{I}(G/H)) \subseteq \sr{I}(G/K)$,
    \item $\tr{K}{H}(\sr{I}(G/K)) \subseteq \sr{I}(G/H)$,
    \item $\N{K}{H}(\sr{I}(G/K)) \subseteq \sr{I}(G/H)$,
    \item $c_{g, H}(\sr{I}(G/H)) = \sr{I}(G/H^{g})$.
\end{enumerate}
Suppose we have a collection of subsets $\mathcal{G}(X)\subseteq T(X)$ for all finite $G$-sets $X$. The \textit{ideal generated by $\mathcal{G}$} is the intersection of all ideals $\mathscr I$ of $T$ with $\mathcal{G}(X)\subseteq \mathscr I(X)$ for all $X$, denoted $\gent{\mathcal{G}}$. If $\mathcal{G}=\{a_1,\dots,a_n\}$ is finite, then we write $\gent{\mathcal{G}}=\gent{a_1,\dots,a_n}$.
\end{definition}

Much like their ring-theoretic counterparts, Tambara ideals can provide insight into algebraic structure (see, for example, \cite[\S 4]{nakaoka:2011}). The definition of a prime Tambara ideal we give below is not the original definition of Nakaoka, but is equivalent \cite[Proposition 4.4]{nakaoka:2011}.

\begin{definition}\label{def:prime}
An ideal $\sr{P}$ of a Tambara functor ${T}$ is \emph{prime} if for any pair of ideals $\sr{I}$ and $\sr{J}$,
$$\sr{I}\sr{J} \subseteq \sr{P} \implies \sr{I}\subseteq\sr{P} \text{ or } \sr{J}\subseteq \sr{P},$$
where the product of ideals is defined as 
$$\sr{I}\sr{J}(X) = \{f_+(ab) \mid  f\in G\mathbf{Fin}(A, X) \text{ for some } A,\; a \in \sr{I}(A),\; b\in \sr{J}(A) \}.$$
\end{definition}

In practice, it can be difficult to prove a given Tambara ideal satisfies this condition. Nakaoka develops a criterion for determining when an ideal is prime in \cite[Proposition 4.2]{nakaoka:2014}, which we shall use extensively throughout this work.

\begin{definition}
\label{def:Q}
Let $\sr{I}$ be an ideal of $T$ and $H, H' \leq G$. Let $a \in T(G/H)$ and $b \in T(G/H')$. Define the proposition $Q(\sr{I}, a, b)$ by   
\begin{enumerate}
    \item[(*)] The relation
    $$\Big(\N{K^g}{L}\circ c_{g, K}\circ \res{K}{H}(a)\Big)
    \cdot\left(\N{K'^{g'}}{L} \circ c_{g', K'} \circ \res{K'}{H'}(b)\right) \in \sr{I}(G/L)$$
    holds for all $L, K, K' \leq G$ and $g, g' \in G$ satisfying $L \geq K^g, ~L \geq K'^{g'}, ~K\leq H,$ and $K'\leq H'$.
\end{enumerate}
\end{definition}

\begin{theorem}[{\cite[Proposition 4.2]{nakaoka:2014}}] \label{thm:nak_prime}
An ideal $\sr{P}$ of a Tambara functor $T$ is prime if and only if for any $a \in T(G/H)$ and $b \in T(G/H')$, the statement $Q(\sr{P}, a, b)$ implies
$$a\in \sr{P}(G/H) \text{ or } b\in \sr{P}(G/H').$$ 
\end{theorem}

For more details on Tambara ideals, we point the reader to \cite{nakaoka:2011}. The theory of Tambara ideals mirrors the theory of ideals of commutative rings. In particular the \emph{spectrum} of a Tambara functor is defined as the collection of all prime ideals under the Zariski topology. 

\begin{definition}\label{defn:spec}
The spectrum of a Tambara functor $T$, denoted $\Spec(T)$, is the collection of all prime ideals $\sr{P}$ of $T$. The topology on $\Spec(T)$ is given by the collection of closed sets of the form 
$$V(\sr{I}) = \{\sr{P} \in \Spec(T) \mid \sr{I} \subseteq \sr{P} \}$$
for all ideals $\sr{I}$ of $T$. 
\end{definition}

 In {\cite[Theorem 4.9]{nakaoka:2011}}, Nakaoka shows that a given surjective Tambara functor morphism $\varphi\colon T\to S$ induces a homeomorphism between $\Spec(S)$ and $V(\ker(\varphi))$ given by 
 $$\varphi_\#: V(\ker(\varphi))\rightarrow \Spec(S)$$
 $$\sr{P} \mapsto \varphi(\sr{P}).$$

\subsection{The Burnside ring, its prime spectrum, and its Tambara structure}
We are interested in determining the collection of prime ideals of the \textit{Burnside Tambara functor} $\burn$ on a finite group $G$. This section recalls some relevant definitions and results, in particular the work of A.~Dress \cite{dress:1971} on the spectrum of the Burnside ring on a finite group. 

\begin{definition} \label{defn:burnside ring}
The \textit{Burnside ring} of a finite group $G$ is the Grothendieck construction on the semi-ring of finite $G$-sets, denoted $A(G)$. That is, $A(G)$ is the ring of formal differences of isomorphism classes of finite $G$-sets, with addition given by disjoint union and multiplication given by Cartesian product.
\end{definition}

In \cite{dress:1971}, Dress shows that the
\textit{mark of $H$ on $X$}
\begin{equation}\label{eqn:mark}
\begin{aligned}
    \varphi_G^H\colon A(G) &\rightarrow \ZZ \\
    X &\mapsto \abs{X^H}
\end{aligned}
\end{equation}
is a ring homomorphism for all $H\leq G$. Here $X^H$ denotes the $H$-fixed points of $X$.  One can verify that the following claim holds:

\begin{lemma}\label{lem:phi_vals}
Let $G$ be an Abelian group and $H, I\leq G$. Then \begin{equation*}
\varphi_G^I(G/H) =  \left\{\begin{array}{cc}
         |G:H| & I \leq H \\
         0 & \text{otherwise}
         \end{array}\right.
\end{equation*}
\end{lemma}

%Proof:
%Suppose $I \leq H$. Then for each $gH \in G/H$ and $i\in I$, $igH = g(iH) = gH$. Similarly, if $I \not\leq H$ then there exists and $i\in I$ such that $iH \neq H$ and therefore $gH \neq g(iH)$. 

Moreover, for an arbitrary finite group $G$, a result due to Burnside  shows that the map $\varphi_G\colon A(G) \rightarrow \prod_{H\leq G} \ZZ$ given by \begin{equation}\label{eqn:phiG}
    \varphi_G = \prod_{H\leq G} \varphi_G^H
\end{equation} is an injective ring homomorphism.

\begin{definition}\label{def:phi}
Given a prime (or zero) $p$, let $\varphi_{G, p}^H\colon A(G) \rightarrow \ZZ/p\ZZ$ denote the composition of $\varphi_G^H$ and the usual quotient map $\ZZ \rightarrow \ZZ/p\ZZ$. 
\end{definition}

\begin{theorem}[{\cite[\S 5]{dress:1971}}] \label{thm:spec burn ring}
The Zariski spectrum of the Burnside ring is $$\Spec(A(G)) = \{\ker\varphi_{G, p}^H\mid H\leq G, ~p \text{ a prime or zero}\}.$$
\end{theorem}

%\begin{example}
%Burnside ring on cyclic gp?
%\end{example}

Dress also formulates some containment results, which are nicely summarized in \cite[Theorem 3.6]{balmer/sanders:2017}.

\begin{theorem}\label{thm:dress containment}
Let $H,K\leq G$ and $p,q$ be primes. Then\begin{enumerate}
    \item[(i)] $\ker(\varphi_{G,0}^H)\subseteq \ker(\varphi_{G,0}^{K})$ if and only if $\ker(\varphi_{G,0}^H)= \ker(\varphi_{G,0}^{K})$ if and only if $H=K^g$ for some $g\in G$;
    \item[(ii)] $\ker(\varphi_{G,p}^H)\subseteq \ker(\varphi_{G,q}^{K})$ implies $p=q$ and $\ker(\varphi_{G,p}^H)= \ker(\varphi_{G,p}^{K})$;
    \item[(iii)] $\ker(\varphi_{G,0}^H)\subseteq \ker(\varphi_{G,p}^{K})$ if and only if $\ker(\varphi_{G,p}^H)= \ker(\varphi_{G,p}^{K})$ if and only if $O^p(H) = (O^p(K))^g$ for some $g\in G$, where $O^p(H)$ is defined below in \cref{defn:OpH};
    \item[(iv)] $\ker(\varphi_{G,0}^H)\subset \ker(\varphi_{G,p}^{H})$ and $\ker(\varphi_{G,p}^H)\not\subseteq \ker(\varphi_{G,0}^{K})$.
\end{enumerate}
\end{theorem}

\begin{definition}\label{defn:OpH}
For any subgroup $H \leq G$ and $p$ prime, let $O^p(H)$ denote the  unique normal $p$-perfect subgroup such that the quotient $H/O^p(H)$ is a $p$-group. That is,
\[
    O^p(H) = \bigcap I,
    \]
where the intersection is over all $I\trianglelefteq H$ such that $|H:I| = p^n$ for some $n\in\NN$.
\end{definition}

\begin{remark}\label{rmk:OpH when G Ab}
When $G$ is Abelian, $O^p(H)$ is the part of $H$ which is relatively prime to $p$.
\end{remark}

\begin{remark}\label{rmk:ker H and OpH}
Note that $X^H=(X^{O^p(H)})^{H/O^p(H)}$ and so \[
\abs{X^H} = \abs{(X^{O^p(H)})^{H/O^p(H)}} \equiv \abs{X^{O^p(H)}} \pmod{p}
\] since $H/O^p(H)$ is a $p$-group. This implies $\ker(\phi_{G,p}^H)= \ker(\phi_{G,p}^{O^p(H)})$.
\end{remark}

Our goal is to extend these results to the context of Tambara functor theory.
When equipped with appropriate transfer, norm, and restriction maps, the Burnside ring admits a Tambara structure, and in fact is the initial object in the category of Tambara functors.

\begin{definition}[The Burnside Tambara functor] \label{defn:burnside tamb functor}
For each $H\leq G$, set $\burn_G(G/H)=A(H)$. For subgroups $K\leq H\leq G$, $g\in G$, $Y\in \burn_G(G/H)$, and $X\in\burn_G(G/K)$, we define the Tambara structure maps,\begin{align*}
    \res{K}{H}\colon \burn_G(G/H)&\longrightarrow \burn_G(G/K)\\
    Y &\longmapsto Y\text{ with restricted $K$-action,}\\
    \tr{K}{H}\colon \burn_G(G/K)&\longrightarrow \burn_G(G/H)\\
    X &\longmapsto H\times_K X,\\
    \N{K}{H}\colon \burn_G(G/K)&\longrightarrow \burn_G(G/H)\\
    X &\longmapsto {\rm Map}_K(H,X),\\
    c_{g,H}\colon \burn_G(G/H)&\longrightarrow \burn_G(G/H^g)\\
    X &\longmapsto X^g.
\end{align*}
These definitions are only valid for actual $G$-sets, as opposed to arbitrary elements of the Burnside Rings which may be formal differences of $G$-sets. However, they can be extended to the entirety of the respective Burnside Rings via a procedure due to Dress which can be found in \cite{curtis:1990}. When $G$ is Abelian, conjugation is trivial. These maps turn $\burn_G$ into the initial $G$-Tambara functor (cf. \cite{nakaoka:2014, tambara:1993}). We will denote $\burn_G$ by merely $\burn$ when the group is obvious from context.
\end{definition}

\begin{remark}\label{rmk:isom AG and AH}
For $K\leq H\leq G$, there is an isomorphism $\burn_G(G/K)\cong \burn_H(H/K)$, and we will often identify the two through this natural isomorphism (cf. \cite[Remark 1.5]{nakaoka:2014}).
\end{remark}

In general, it can be difficult to find simplified descriptions for the Tambara structure maps.
However, when $G$ is Abelian, we can directly obtain a formula for the norm in $\burn_G$. The following formula is due to Nakaoka in \cite[Appendix]{nakaoka:2014}.

\begin{theorem}[Nakaoka Norm Formula]
\label{thm:nak_norm}
Let $G$ be an Abelian group and $X = \sum_{I\leq H} m_I H/I \in \burn_G(G/H)$. Then 
$$\N{H}{G}(X) = \sum_{K \leq G} \frac{C(K)}{\abs{G:K}}G/K,$$
where $C(K)$ is defined inductively by
$$C(K) = \left(\sum_{K\cap H \leq I \leq H} m_I\abs{H:I}\right)^{\abs{G:KH}} - \sum_{K < L \leq G}C(L),$$ with the base case $C(G)=m_H$. 
\end{theorem}

%--------------------------------------------------------------%--------------------------------------------------------------
%--------------------------------------------------------------%--------------------------------------------------------------

\section{A Collection of Prime Ideals}\label{sec:collection of prime ideals}

Remarkably, by intersecting the ideals found by Dress, one can derive a family of prime ideals of the Burnside Tambara functor which is also indexed by pairs consisting of a subgroup and a number that is either prime or $0$.

\begin{definition}\label{def:the_ideals}
Let $G$ be a finite group and $\burn_G$ the Burnside Tambara functor on $G$. Let $C$ be a subgroup of $G$ and $p$ a prime or zero. For each $H\leq G$ define 
$$\fr{p}_{C, p}(G/H) = \bigcap_{\substack{I \leq H \\ I \preccurlyeq_G C}} \ker(\varphi_{H, p}^I)$$
where $I \preccurlyeq_G C$ means that there exists some $g \in G$ such that $I^g \leq C$. 
\end{definition}

The remainder of this section will be devoted to proving that for all $C\leq G$ and $p$ a prime or zero, $\fr{p}_{C, p}$ is a prime ideal of $\burn_G$. This goal will be easier to achieve by switching our perspective on the Burnside Tambara functor. Namely, we will use the $\varphi_H$ from \cref{eqn:phiG} to construct a monomorphism from $\burn_G$ to another Tambara functor and study the $\mathfrak{p}_{C,p}$ in the context of this other Tambara functor instead. 

As a first step, we observe how the mark homomorphisms interact with the structure maps of the Burnside Tambara functor. These formulas will guide our formulation of the structure maps for the new Tambara functor.

\begin{lemma}\label{lem:phi_res}
Let $X$ be a finite $H$-set for $H\leq G$. Then \begin{enumerate}
    \item[(i)] for all $I\leq K\leq H$, we have $\varphi_K^I(\res{K}{H}(X)) = \varphi_H^I(X)$;
    \item[(ii)] for all $g \in G$ and $I\leq H^g$, we have $\varphi^I_{H^g}(c_{g, H}(X)) = \varphi_H^{I^{g^{-1}}}(X)$;
    \item[(iii)] for all $I\leq G$, we have $$\varphi_G^I(\tr{H}{G}(X)) = \sum_{[g] \in G/H,\; I^g\leq H}\varphi_H^{I^g}(X);$$
    \item[(iv)] for all $I\leq G$, we have
\[\varphi_G^I(\N{H}{G}(X)) = \prod_{[g] \in I\setminus G/H} \varphi_H^{I^g \cap H}(X)
%for Ab: (\varphi_{H}^{I\cap H}(X))^{|G:\gen{I\cup H}|}.
\]
\end{enumerate}
\end{lemma}
\begin{proof}
The first two claims are straightforward to verify. 
%The restriction of $X$ is simply the same underlying set with action restricted to $H$, so the result clearly follows. 
%By definition $c_{g, H}(X) = \{g^{-1} x\;|\; x\in X\}$. It is clear that an element of this set $g^{-1}x$ is fixed by $K^g = \{g^{-1}kg\;|\; k\in K\}$ for some subgroup $K \leq H$ precisely when $x$ is fixed by $K$. The result follows by setting $K = I^{g^{-1}}$. 
To show the third, note that by definition, $\tr{H}{G}(X) = G\times_H X$ with a $G$-action given by $g\cdot [(g', x)] = [(gg', x)]$ for all $g, g' \in G$ and $x\in X$. Suppose some element $[(g, x)]$ is fixed by $I$. Then for all $i\in I$, we have that $[(ig, x)] = [(g, x)]$, so there is some $h\in H$ such that both $igh^{-1} = g$ and $h\cdot x = x$. Now the first condition says that $h =  g^{-1}ig$, therefore we must have that $I^g \leq H$ and $x$ is fixed by $I^g$. These are necessary conditions and are also clearly sufficient, thus the fixed points of $I$ are precisely the elements of the form $(g, x)$ where $I^g \leq H$ and $x$ is fixed by $I^g$. Since each equivalence class has $\abs{H}$ elements, there are precisely
$$\left(\sum_{\substack{g \in G \\ I^g\leq H}}\varphi_H^{I^g}(X)\right)/\abs{H} $$
such elements. Now, if $I^g \leq H$ then $I^{gh}\leq H$ for all $h\in H$, so we can instead sum over the cosets $G/H$ and obtain the expression in the theorem statement. This proves part (iii). The formula for (iv) follows from adjointness and the Mackey formula (see \cite[p.39]{yoshida:1990}).
\end{proof}

It is known (cf. \cite{boltje/danz:2012}) that the Burnside ring $A(G)$ can be identified with the image of $\varphi_G\colon A(G)\to \prod_{H\leq G}\ZZ$. In particular, the image of $\varphi_G$ is contained in the \textit{ghost ring} of $A(G)$, often denoted $\tilde A(G)$, which is the sub-ring of $G$-fixed points of $\prod_{H\leq G}\ZZ$ (where the $G$-action permutes the coordinates according to the conjugation action of $G$ on its subgroups). Is it this ghost ring that we will generalize to the context of Tambara functors. We note that this idea (and much more) is discussed in \cite{strickland:2012}, and in particular \cite[\S 19]{strickland:2012} covers connections between Tambara functors and the Witt vector ghost map.

\begin{definition}\label{cor:structure map fmlas}\label{def:ghost} %We will first define an analogy of $\prod_{H\leq G}\ZZ$. Let $\prod_G \underline{\ZZ}$ be the Tambara functor with $\prod_G \underline{\ZZ}(G/H) = \prod_{I\leq H}\ZZ$.
Just as with Burnside ring, the ghost rings $\tilde{A}(H)$ for $H\leq G$ assemble into a Tambara functor $\burnghost_G$ which we will call the \textit{ghost Tambara functor} on $\burn_G$. The Tambara structure maps are defined as follows:
For all subgroups $H\leq K \leq G$, elements $a \in \burnghost_G(G/H)$, $b \in\burnghost_G(G/K)$, and $g\in G$, the restriction, transfer, norm, and conjugation maps in $\burnghost_G$ are given by
\begin{align*}
    \res{H}{K}(b) &= (b_I)_{I\leq H}\\
    \tr{H}{K}(a) &= \left(\sum_{\substack{[k] \in K/H\\ I^k\leq H}}a_{I^k}\right)_{I\leq K}\\
    %\N{H}{K}(a) &= \left(a_{I\cap H}^{\abs{K:\gen{I\cup H}}}\right)_{I\leq K}\\
    \N{H}{K}(a) &= \left(\prod_{[g] \in I\setminus K/H}
    a_{I^g \cap H}\right)_{I\leq K}\\
    c_{g, H}(a) &= \left(a_{I^{g^{-1}}}\right)_{I\leq H^g}
\end{align*}
The fact that this defines a Tambara functor follows from the observation that $\burnghost_G$ lives inside the fixed-point Tambara functor on the ring $\prod_{K\leq G}\ZZ$. Specifically, we can extend a tuple $(a_I)_{I\leq H}$ in $\burnghost_G(G/H)$ to a tuple in $(\prod_{K\leq G}\ZZ)^H$ by inserting $0$s at all the $K$-coordinates for $K\not\leq H$, and since the structure maps agree we get all the Tambara structure for free. 
\end{definition}
\begin{remark}
For an Abelian group, the above expressions simplify greatly:
\begin{align*}
    \res{H}{K}(b) &= (b_I)_{I\leq H},\\
    \tr{H}{K}(a) &= (|K:H|a_I)_{I\leq K},\\
    \N{H}{K}(a) &= \left(a_{I\cap H}^{|K:HI|}\right)_{I\leq K},\\
    c_{g, H}(a) &= a.
\end{align*}
\end{remark}

Let $\varphi = \{\varphi_H\}_{H\leq G}$, where $\varphi_H$ is as defined in \cref{eqn:phiG}. Then $\varphi\colon \burn_G\to \burnghost_G$ defines an injective morphism of Tambara functors, and we can identify $\burn_G$ with its image, which we will denote by $\burn_G^\varphi$.

\begin{remark}\label{thm:iso_rep}
We can do this sort of construction more generally for any Tambara functor $T$ on a group $G$. Given a collection $f = \{f_H\}_{H\leq G}$ of injective ring homomorphisms $f_H\colon T(G/H) \rightarrow R_H$ for some rings $R_H$, we can build a new Tambara functor $T^f$. For each $H\leq G$ define $T^f(G/H) = {\rm im}(f_H)$, and for each $H\leq K\leq G$ and $g\in G$ define maps
\begin{align*}
    \res{H}{K} &= f_H\circ\res{H}{K}\circ f_K^{-1}\\
    \tr{H}{K} &= f_K\circ\tr{H}{K}\circ f_H^{-1}\\
    \N{H}{K} &= f_K\circ\N{H}{K}\circ f_H^{-1}\\
    c_{g, H} &= f_{H^g}\circ c_{g, H} \circ f^{-1}_H\\
\end{align*} 
where res, tr, N, and $c$ on the right refer to the restriction, transfer, norm and conjugation maps in $T$. By pulling the necessary commutative diagrams outlined in \cite{mazur:2019} back along these isomorphisms, we get that $T^f$ is a Tambara functor and $f\colon T\rightarrow T^f$ is a Tambara functor isomorphism.
\end{remark}

Given  $C\leq G$ and $p$ a prime or zero define $\sr{K}_{C, p}$ by 
$$\sr{K}_{C, p}(G/H) = \left(\prod_{I\leq H} \chi_{C, p}(I)\right ) \cap \tamb{A}_G^\varphi(G/H)$$
where
$$\chi_{C, p}(I) = \left\{\begin{array}{cc} (p) & I \preccurlyeq_G C \\
         \ZZ & {\rm otherwise.}
         \end{array}\right.$$

\begin{remark}\label{rmk:P and K equiv}
Essentially by construction, we have $\varphi(\fr{p}_{C, p}) = \sr{K}_{C, p}$. Hence to show that $\fr{p}_{C, p}$ is a prime ideal of the Burnside Tambara functor, we can instead prove the equivalent statement that $\sr{K}_{C, p}$ is a prime ideal of $\burn_G^\varphi$.
\end{remark}

\begin{proposition}
Let $C\leq G$ and $p$ prime or zero. Then $\fr{p}_{C, p}$ is an ideal of $\burn_G$. 
\end{proposition}

\begin{proof}
By \cref{rmk:P and K equiv}, this is equivalent to showing that $\sr{K}_{C, p}$ is an ideal of $\burn_G^\varphi$. To prove that $\sr{K}_{C,p}$ is closed under the structure maps, we will use the formulas from \cref{cor:structure map fmlas}. Let $H\leq K \leq G$, $a \in\sr{K}_{C, p}(G/K)$, and $g\in G$. Then for all $I\leq H$ satisfying $I\preccurlyeq_G C$ we have that
$$\left(\res{H}{K}(a)\right)_I = a_I \in (p)$$
implying $\res{H}{K}(a) \in \sr{K}_{C,p}(G/H)$.

Now suppose $I\leq K$ satisfies $I\preccurlyeq_G C$. Then $$\left(\tr{H}{K}(a)\right)_I = \left(\sum_{\substack{[k] \in K/H\\ I^k\leq H}}a_{I^k}\right) \in (p),$$ since $I^k\preccurlyeq_G C$ implies $a_{I^k}\in (p)$ for each $[k]\in K/H$, by the assumption on $a$. Moreover, by definition of $I\preccurlyeq_G C$, there exists $\tilde g\in G$ such that $I^{\tilde g}\leq C$, and so $I^{\tilde g}\cap H\leq C$ as well which means $a_{I^{\tilde g}\cap H}\in (p)$ by assumption on $a$. Thus\[
\left(\N{H}{K}(a)\right)_I = \prod_{[g]\in I\setminus K/H} a_{I^g\cap H} = \left(\prod_{[g]\neq [\tilde g]} a_{I^g\cap H}\right)\cdot a_{I^{\tilde g}\cap H}
\] is in $(p)$ as well.

Finally, for all $I\leq H^g$ satisfying $I\preccurlyeq_G C$, we have that
$$\left(c_{g, H}(a)\right)_I = a_{I^{g^{-1}}} \in (p)$$
so $c_{g, H}(a) \in \sr{K}_{C, p}$. Therefore $\sr{K}_{C, p}$ is an ideal of $\burn_G^\varphi$. 
\end{proof}

\begin{theorem}\label{thm:PCp prime}
Let $G$ be a finite group, $C\leq G$ and $p$ prime or zero. Then $\fr{p}_{C, p}$ is a prime ideal of $\burn_G$.
\end{theorem}

\begin{proof}
It is equivalent to show that the ideal $\sr{K}_{C, p}$ is prime in $\burn_G^\varphi$. We will use the criterion in \cref{thm:nak_prime}. Let $H, H' \leq G$, $a \in \burn_G^\varphi(G/H)$ and $b \in \burn_G^\varphi(G/H')$. We will show the contrapositive of \cref{thm:nak_prime}, namely that if $a\not\in \sr{K}_{C,p}(G/H)$ and $b\not\in\sr{K}_{C,p}(G/H')$, then $Q(\sr{K}_{C,p},a,b)$ does not hold. To show this, it suffices to find $K\leq H$, $K'\leq H'$, $g,g'\in G$ and $L\leq G$ such that $K^g\leq L$, $K'^{g'}\leq L'$ and\[
(\N{K^g}{L}\circ c_{g, K}\circ \res{K}{H}(a))\cdot(\N{K'^{g'}}{L}\circ c_{g, K'}\circ \res{K'}{H'}(b)) \not\in \sr{K}_{C, p}(G/L).
\] Since $a\not\in \sr{K}_{C,p}(G/H)$ this means there is some $K\leq H$ which is sub-conjugate to $C$ such that $a_K\not\in (p)$.

We can use the formulas from \cref{cor:structure map fmlas} to see that in fact the $C$-coordinate of $$\N{K^g}{L}\circ c_{g, K}\circ \res{K}{H}(a)$$ is not in $(p)$. Specifically, if we choose $g\in G$ such that $K^g\leq C$, then the $K^g$-coordinate of $c_{g,K}(\res{K}{H}(a))$ is $a_{(K^g)^{g^{-1}}}=a_K\not\in (p)$. After we norm up to $C$, the $C$-coordinate is\[
(\N{K^g}{C}\circ c_{g, K}\circ \res{K}{H}(a))_C = \prod_{[\sigma]\in C\setminus C/K^g} (c_{g,K}(\res{K}{H}(a)))_{C^\sigma\cap K^g},
\] which is just $(c_{g,K}(\res{K}{H}(a)))_{K^g}=a_K$ (since $C\setminus C/K^g=*$), and hence is not in $(p)$. An identical argument shows that there is $K'\leq H'$ and $g'\in G$ such that $K'^{g'}\leq C$ and $(\N{K'^{g'}}{C}\circ c_{g', K'}\circ \res{K'}{H'}(b))_C\not\in (p)$. Thus the $C$-coordinate of the product \[
(\N{K^g}{L}\circ c_{g, K}\circ \res{K}{H}(a))\cdot(\N{K'^{g'}}{L}\circ c_{g, K'}\circ \res{K'}{H'}(b))
\] is not in $(p)$.
Since\[
\sr{K}_{C,p}(G/C) = \prod_{I\leq C} (p) \cap \underline{A}^{\varphi}_G(G/C),
\] this means $(\N{K^g}{L}\circ c_{g, K}\circ \res{K}{H}(a))\cdot(\N{K'^{g'}}{L}\circ c_{g, K'}\circ \res{K'}{H'}(b))$ cannot be in $\sr{K}_{C,p}(G/C)$.

\end{proof}

Comparing to Dress's theorem on the inclusion structure of the prime spectrum of the Burnside ring (\cref{thm:dress containment}), we have the following containment rules:

\begin{theorem}\label{thm:containment}
Let $G$ be Abelian, $H,K\leq G$ and $p,q$ prime. Then we have\begin{enumerate}
    \item[(i)] $\fr{p}_{H, 0}\subseteq \fr{p}_{K, 0}$ if and only if $K \leq H$,
    \item[(ii)] $\fr{p}_{H, 0}\subset \fr{p}_{H, p}$ and $\fr{p}_{H, p}\not\subseteq \fr{p}_{K, 0}$,
    \item[(iii)] $\fr{p}_{H, p}\subseteq \fr{p}_{K, q}$ if and only if $p = q$ and $O^p(K)\leq O^p(H)$,
\end{enumerate} where $O^p(H)$ is as in \cref{defn:OpH}.
\end{theorem}
\begin{proof}
(i) The ``if'' direction follows directly from \cref{thm:dress containment} and the definition of $\fr{p}_{H, p}$ as the intersection of the kernels of the mark maps. To see the ``only if'' direction, suppose that $\fr{p}_{H,0}\subseteq \fr{p}_{K,0}$. This implies, in particular, that\[
\fr{p}_{H,0}(G/G) =\bigcap_{I\leq H} \ker\varphi_G^I \subseteq   \bigcap_{I\leq K} \ker\varphi_G^I = \fr{p}_{K,0}(G/G).
\] Consider the $G$-set $X=\abs{G}\cdot G/G-\abs{H}\cdot G/H$. By \cref{lem:phi_vals}, the value of $X$ under $\varphi^I_G$ is\begin{equation}\label{eqn:value of X}
    \varphi_G^I(X) = \left\{\begin{array}{cc}
     0 &  I\leq H;\\
     \abs{G} & \text{otherwise}.
\end{array}\right.
\end{equation}
Thus $X\in \fr{p}_{H,0}(G/G)\subseteq \fr{p}_{K,0}(G/G)$, which means that $\varphi_G^I(X)=0$ for all $I\leq K$. Taking $I=K$ in \cref{eqn:value of X} shows $K\leq H$ as desired.

The first part of (ii) follows directly from the definition of $\fr{p}_{H, p}$. The second follows from the fact that for all $H, K\leq G$, $pG/G\in\fr{p}_{H, p}(G/G)$ and $pG/G\not\in\fr{p}_{K, 0}(G/G)$ since $\varphi^I_G(p G/G) = p$ for all $I \leq G$.  

 To show (iii), first suppose that $p=q$ and $O^p(K)\leq O^p(H)$. We claim that
\begin{equation}\label{eqn:pHp}
    \fr{p}_{H, p}(G/L) = \bigcap_{I \leq L\cap H}\ker\varphi_{L, p}^I = \bigcap_{I \leq L\cap H}\ker\varphi_{L, p}^{O^p(I)} = \bigcap_{J \leq O^p(L)\cap  O^p(H)}\ker\varphi_{L, p}^J,
\end{equation} for all $L\leq G$. The first equality is by definition, and the second follows from \cref{rmk:ker H and OpH}. For the final equality, we note that\[
\{O^p(I) \mid I\leq L\cap H\} = \{J\mid J\leq O^p(L)\cap O^p(H)\},
\] because both sets consist of precisely the subgroups of $L\cap H$ which are relatively prime to $p$ (see \cref{rmk:OpH when G Ab}).
%The forward inclusion follows from the observation that for any $I,J\leq G$, we have $O^p(I\cap J) = O^p(I)\cap O^p(J)$. Moreover, if $I\leq J$, then $O^p(I) \leq O^p(J)$. To show the reverse inclusion, suppose $J \leq O^p(L\cap H)$. Then $\text{Syl}_p(J) = e$ implying $O^p(J) = J \leq L\cap H$ and therefore $J\in \{O^p(I) \mid I\leq L\cap H\}$.
The inclusion $\fr{p}_{H, p} \subseteq \fr{p}_{K, p}$ follows from \cref{eqn:pHp}. 

For the other implication of (iii), note that if $\fr{p}_{H,p}\subseteq \fr{p}_{K,q}$ then $p$ must equal $q$. To show that $O^p(K)\leq O^p(H)$, we use essentially the same strategy as in part (i). By \cref{eqn:pHp} and our assumption that $\fr{p}_{H,p}\subseteq \fr{p}_{K,q}$, we know
\[
\fr{p}_{H,p}(G/O^p(G)) =\bigcap_{I\leq O^p(H)} \ker\varphi_{O^p(G),p}^I \subseteq   \bigcap_{I\leq O^p(K)} \ker\varphi_{O^p(G),p}^I = \fr{p}_{K,p}(G/O^p(G)).
\]
Taking $X=\abs{O^p(G)}O^p(G)/O^p(G) - \abs{O^p(H)}O^p(G)/O^p(H)$ we can essentially repeat the argument from part (i). In particular, we have
\begin{equation}\label{eqn:value of X (2)}
    \varphi_{O^p(G),p}^{I}(X) = \left\{\begin{array}{cc}
     0 &  I\leq O^p(H);\\
     \abs{O^p(G)} & \text{otherwise}.
\end{array}\right.
\end{equation}
By construction, $X \in \fr{p}_{H,p}(G/O^p(G))\subseteq \fr{p}_{K,p}(G/O^p(G))$ which implies that $O^p(K) \leq O^p(H)$ by \cref{eqn:value of X (2)}.
\end{proof}

Having found this family of prime ideals of $\burn_G$ and determined their containment structure, we can compare our findings to the structure of $\Spec(A(G))$. Remarkably, while the two sets are in bijection, they are not homeomorphic in general.

\begin{proposition}\label{prop:no homeom}
For an Abelian group $G$, there is a bijection between $\Spec(A(G))$ and 
$$\{\fr{p}_{C,p}\mid C\leq G, ~p\text{ prime or }0\}\subseteq\Spec(\burn_G).$$ However, the two spaces are not homeomorphic (for nontrivial $G$). 
\end{proposition}
\begin{proof}
To see that the spectra are in bijection, recall that $\Spec(A(G))$ is comprised of kernels of maps $\varphi^C_{G,p}$ for $C\leq G$ and $p$ prime or $0$ (see \cref{def:phi} and \cref{thm:spec burn ring}). When $G$ is Abelian, the above shows that there is a bijection between $\{\fr{p}_{C,p}\mid C\leq G, ~p\text{ prime or }0\}\subseteq \Spec(\burn_{G})$ and $\Spec(A(G))$ given by \[
\fr{p}_{C, p} \:\longleftrightarrow \ker(\varphi^C_{G, p}).\]

However, this bijection is not a homeomorphism, and in fact there is no homeomorphism between the two spectra (for $G\neq e$). One way to see this is to consider the Krull dimension of each spectrum. From \cref{thm:containment}, we see that the Krull dimension of $\{\fr{p}_{C,p}\mid C\leq G, ~p\text{ prime or }0\}$ is equal to $m+1$ where $m$ is the maximal length of a subgroup series in $G$. This can be seen by considering the chain of prime ideals 
$$\fr{p}_{G, 0} \subseteq \fr{p}_{G, q} \subseteq \fr{p}_{H_1, q} \cdots  \subseteq \fr{p}_{{H_{m-1}}, q} \subseteq \fr{p}_{H_m, q}$$
where $H_m=e$, each $|H_i:H_{i+1}|$ is prime, and $q \nmid |G|$. However, by theorem \cref{thm:dress containment}, the Krull dimension of $\Spec(A(G))$ is easily seen to be 1 since the only non-trivial inclusion is $\ker(\varphi^H_{G, 0}) \subset \ker(\varphi^H_{G, p})$. 
\end{proof}
%-------------------------%-------------------------%-------------------------%-------------------------
\section{Describing The Ideals In Terms of Generators}\label{sec:gens}

We suspect that such $\fr{p}_{C,p}$ comprise the entire prime spectrum of $\burn_G$. Using different methods, Nakaoka has shown that this is indeed the case for any cyclic $p$-group $C_{p^n}$ \cite[Theorem 6.12]{nakaoka:2014}. A direct reformulation of Nakaoka's theorem statement in terms of our ideals is made in \cref{thm:Nak_main}. In this section we make progress towards expanding Nakaoka's result by describing our ideals in terms of certain generators.

\subsection{Ring-Theoretic Generators}
In this section we will describe the ideals $\fr{p}_{C, p}$ of $\burn_G$ in the case that $G$ is a cyclic group. In order to describe an arbitrary $X\in \fr{p}_{C,p}(G/G)$, it will be helpful to partition $X$ into smaller pieces, denoted $X_J$ for $J\leq G$. The following proposition describes certain functions $\psi^J$ which will help us partition $X$ in a useful way.

\begin{proposition}
Let $G$ be an Abelian group, $C\leq G$, and $p$ be prime or zero. Take $X = \sum_{I\leq G}m_IG/I\in \fr{p}_{C,p}(G/G)$. For each $J \leq C$, define $S_J = \{I\leq G \mid I\cap C = J\}$ and
$$\psi^J(X) = \sum_{I\in S_J}m_I|G:I|.$$
Then $$\psi^J(X) \equiv_p 0.$$
\end{proposition}
\begin{proof}
Let $$X = \sum_{I\leq G}m_IG/I \in \burn(G/G)$$ and suppose that $X \in \fr{p}_{C, p}(G/G)$. Then, by applying \cref{lem:phi_vals} to the condition that $\varphi_G^J(X) \in (p)$, we have that
\begin{equation}
\label{eq:phi_p}
\varphi_{G}^J(X) = \sum_{J\leq I \leq G} |G:I|m_I \in (p),
\end{equation} 
for all $J \leq C$.
For each such $J$, define $\zeta^J(X)$ by the recurrence
\begin{equation*}
\zeta^J(X) = \varphi_G^J(X) - \sum_{J < I\leq C}\zeta^I(X).
\end{equation*}
Then by \cref{eq:phi_p}, we have
$$\zeta^J(X) \equiv_p 0$$
for all $J \leq C$. We will show by induction that
\begin{equation*}
\zeta^J(X) = \psi^J(X).
\end{equation*}
The base case is trivial since $\zeta^C(X) = \varphi^C(X) = \psi^C(X) $. Let $J \leq C$, and suppose that the statement holds for all $I \leq C$ such that $J<I$. Then
\begin{align*}
\zeta^J(X) &= \varphi^J_G(X) - \sum_{J < K \leq C} \zeta^K(X) \\
          &= \sum_{I \geq J} m_I|G:I| - \sum_{J < K \leq C} \psi^K(X) \\
          &= \sum_{I \geq J} m_I|G:I| - \sum_{J < K \leq C} \sum_{H \cap C = K} m_H |G:H|\\
          &= \sum_{I \cap C = J} m_I |G:I|\\
          &= \sum_{I \in S_J} m_I |G:I|\\
          &= \psi^J(X)
\end{align*}
where the last line follows since for any subgroup $J < I \leq G$ we have 
$J \leq C \cap I$. Since $\psi^J(X)=\zeta^J(X)$, it follows immediately that $\psi^J(X) \equiv_p 0$. 
\end{proof}

\begin{remark}\label{rmk:SJ partition}
The collection of $S_J$ as $J$ varies clearly partition the subgroups of $G$. Therefore each $X \in \fr{p}_{C,p}(G/G)$ decomposes as 
$$X = \sum_{J\leq C}X_J$$
where 
$$X_J = \sum_{I\in S_J}m_IG/I$$
and each $X_J$ satisfies 
$$\psi^J(X_J) = \sum_{I\in S_J}m_I|G:I| \equiv_p 0.$$
\end{remark}

\begin{remark}\label{rmk:CS gps}
One property of cyclic groups that is essential to the next theorem is the existence of a unique maximal subgroup $M_J$ for each $S_J$ with respect to inclusion. This observation, while not difficult to verify directly, is related to the fact that cyclic groups belong to the larger family known as Closed Summand (CS) groups (cf. \cite{Tercan}). One of the properties of such groups is that for any $H\leq G$, there is a unique subgroup $M\leq G$ maximal amongst the set of subgroups $I$ such that $H \cap I = e$. Since every quotient of a cyclic group is another cyclic group, and therefore a CS group, cyclic groups satisfy the stronger condition mentioned above.
\end{remark}

\begin{theorem}
\label{thm:first_gen}
Suppose $G$ is a cyclic group and let $p$ be $0$ or prime. Then the ideal $\fr{p}_{C, p}(G/G)$ is generated by the elements 
\begin{enumerate}
    \item $pG/G$,
    \item $G/H$ for all $H\leq G$ such that $p \mid |G:H|$,
    \item $G/K - |M_J:K|G/M_J$ for each $J \leq C$ and each $K\in S_J$.
\end{enumerate}
Note that in the case $p = 0$, this reduces to the generators
\begin{enumerate}
    \item $G/K - |M_J:K|G/M_J$ for each $J \leq C$ and each $K\in S_J$
\end{enumerate}
\end{theorem}

\begin{proof}
Recall that $\fr{p}_{C, p}(G/G)$ is an ideal of $\underline{A}_G(G/G) = A(G)$, since $\fr{p}_{C, p}$ is a Tambara ideal. A quick check verifies that all these elements are indeed members of $\fr{p}_{C, p}(G/G)$. Now let $X \in \fr{p}_{C, p}(G/G)$. By \cref{rmk:SJ partition}, it suffices to show that $X_J$ can be written in terms of the above generators for any given $J\leq C$. So we now consider 
$$X_J = \sum_{I\in S_J}m_IG/I,$$
knowing that
$$\sum_{I\in S_J}m_I|G:I| \equiv_p 0.$$
By \cref{rmk:CS gps}, each $I\in S_J$ is necessarily contained in some unique $M_J$ (where $M_J$ is maximal in $S_J$) and therefore $\abs{G:I}$ is divisible by $|G:M_J|$. If $p\mid \abs{G\colon M_J}$, then $p\mid \abs{G\colon I}$ as well, and so $X_J$ may be written in terms of generators of the form (2). Otherwise, we may safely divide by $\abs{G\colon M_J}$ and the equation above becomes
$$m_{M_J} \equiv_p -\sum_{I\in S_J\setminus\{M_J\}} m_I\abs{M_J\colon I}$$
and therefore
$$X_J = \sum_{I\in S_J\setminus \{M_J\}}m_I(G/I - |M_J:I|G/M_J) + npG/M_J$$
for some $n\in\ZZ$. 
\end{proof}

\begin{remark}\label{rmk:pep gens}
When $C=e$ or $C=G$, we can extend this result to any finite Abelian group $G$. In the case $C=e$, it suffices to consider $S_e$, which is the set of all subgroups of $G$, and therefore $M_e=G$. Then the argument above shows that the ideal $\fr{p}_{e, p}$ is generated by
\begin{enumerate}
    \item $pG/G$,
    \item $G/H$ for all $H\leq G$ such that $p \mid |G:H|$,
    \item $G/K - |G:K|G/G$ for each $K\in G$.
\end{enumerate}
In the case $C=G$, note that $S_J = { J }$ for each $J \leq G$, and therefore $M_J=J$. Then the argument above shows that the ideal $\fr{p}_{G, p}$ is generated by
\begin{enumerate}
    \item $pG/G$,
    \item $G/H$ for all $H\leq G$ such that $p \mid |G:H|$
\end{enumerate}

\end{remark}

\subsection{Tambara-Theoretic Generators}
We now wish to develop a list of Tambara functor theoretic generators for our prime ideals. These generators will be helpful when we build up prime ideals of $\burn_G$ from prime ideals of $\burn_H$ for $H\leq G$ (using the isomorphism from \cref{rmk:isom AG and AH}). %Before leveraging this relationship, we prove a necessary lemma.

\begin{theorem} \label{thm:simple_gens}
Let $G$ be a finite Abelian group and $\sr{P}$ be a prime ideal of $\burn_G$. If $pe/e\in \sr{P}(G/e)$ then $pH/H \in \sr{P}(G/H)$ for all $H\leq G$. 
\end{theorem}
\begin{proof}
We perform induction on the order of $H$. Consider a subgroup $e<H \leq G$ and assume the result for all subgroups of smaller order. Note that \[
\N{e}{H}(pe/e) = {\rm Map}(H, pe/e)
\] is defined to be the set of non-equivariant maps $H \rightarrow pe/e$ with $H$-action given by conjugation. Since $pe/e$ has trivial $H$-action, this is the same as pre-composition by multiplication in $H$. As an $H$-set, the norm decomposes into
\begin{equation*}
    \N{e}{H}(pe/e) = pH/H + \sum_{I < H}a_I H/I
\end{equation*}
where the coefficients $a_I=\frac{C(I)}{\abs{G:I}}$ are determined inductively via \cref{thm:nak_norm}. Therefore,
\begin{align*}
    p^2H/H &= p\N{e}{H}(pe/e) - \sum_{I < H}a_I pH/I \\
           &= p\N{e}{H}(pe/e) - \sum_{I < H}a_I \tr{I}{H}(pI/I)
\end{align*}
implying that $p^2H/H \in \sr{P}(G/H)$ by the inductive hypothesis.

Since $\sr{P}$ is prime, in order to conclude that $pH/H\in \sr{P}(G/H)$, it is sufficient to show $Q( \sr{P}, pH/H, pH/H)$ holds. Since $G$ is Abelian, this means proving that for all $L, J, J' \leq G$ satisfying $L \geq J$ and $L \geq J'$, $J \leq H$ and $J' \leq H$ we have
\begin{equation*}\left(\N{J}{L}\circ \res{J}{H}(pH/H)\right) \cdot \left(\N{J'}{L}\circ \res{J'}{H}(pH/H)\right) \in \sr{P}(G/L)\end{equation*}
which simplifies to
\begin{equation*}\N{J}{L}(pJ/J)\cdot \N{J'}{L}(pJ'/J') \in \sr{P}(G/L).\end{equation*}

If we suppose that $J = J' = H$, then 
$$\N{J}{L}(pJ/J) \cdot \N{J'}{L}(pJ'/J') = \N{H}{L}(p^2H/H)$$
must be in $\sr{P}(G/L)$, since $p^2H/H \in \sr{P}(G/H)$. Now, without loss of generality, suppose $J < H$. Then the inductive hypothesis tells us that $pJ/J \in \sr{P}(G/J)$ which implies that $\N{J}{L}(pJ/J) \cdot \N{J'}{L}(pJ'/J') \in \sr{P}(G/L)$ by the ideal properties. Therefore $Q( \sr{P}, pH/H, pH/H)$ holds, implying $p H/H \in \sr{P}(G/H)$.
\end{proof}

%\begin{theorem}Let $G$ be a cyclic group and $\sr{P}$ be a prime ideal of $\burn_G$. If $pe/e\in \sr{P}(G/e)$ then $pH/H \in \sr{P}(G/H)$ for all $H\leq G$. \end{theorem}
%\begin{proof}
%We perform induction on the order of $H$: Consider $H\leq G$ with some order greater than 1 and assume the result for all subgroups of smaller order. Then, since $G$ is Abelian, there is some $K\leq H$ with $|H:K| = q$ prime. If $q \neq p$ then we apply \cref{lem:simple_gen_lem} and we are done. Otherwise, suppose that $q = p$. By the inductive hypothesis, we have that $pK/K \in \sr{P}(G/K)$. Recall from the proof of \cref{lem:simple_gen_lem} that
%\begin{align*}    \tr{K}{H}(pK/K) &= pH/K\\    \N{K}{H}(pK/K) &= pH/H + (p^{p-1}-1)H/K, \\\end{align*}
%and therefore $$p^2(H/H) = p\N{K}{H}(pK/K) - (p^{p-1}-1)\tr{K}{H}(pK/K) \in \sr{P}(G/H).$$ Since $\sr{P}$ is prime, in order to conclude that $pH/H\in \sr{P}(G/H)$, it is sufficient to show $Q( \sr{P}, pH/H, pH/H)$ holds. This means proving that $$\N{J}{L}(pJ/J) \cdot \N{J'}{L}(pJ'/J') \in \sr{P}(G/L)$$ for all $J,J',L$ such that $J, J' \leq H$ and $L \geq J, J'$. If we suppose that $J = J' = H$, then  $$\N{J}{L}(pJ/J) \cdot \N{J'}{L}(pJ'/J') = \N{H}{L}(p^2H/H)$$ must be in $\sr{P}(G/L)$, since $p^2H/H \in \sr{P}(G/H)$. Now, without loss of generality, suppose $J < H$. Then the inductive hypothesis tells us that $pJ/J \in \sr{P}(G/J)$ which implies that $\N{J}{L}(pJ/J) \cdot \N{J'}{L}(pJ'/J') \in \sr{P}(G/L)$ by the ideal properties. Therefore $Q( \sr{P}, pH/H, pH/H)$ holds, implying $p H/H \in \sr{P}(G/H)$.
%\end{proof}

Note that \cref{thm:simple_gens} has the immediate corollary that if $pe/e\in \sr{P}(G/e)$, then $pH/K\in \sr{P}(G/H)$ for all $K\leq H\leq G$. This follows from the multiplication laws in the Burnside ring: $(pH/H)\cdot(H/K) = pH/K$.

For any prime ideal $\sr{P}$ of $\burn_G$, $\sr{P}(G/e)$ is a prime ideal of $\ZZ$ in the ring theoretic sense \cite[Prop. 2.6]{nakaoka:2012}. For $p$ a prime or zero we say that $\sr{P}$ is an \emph{ideal over $p$} if $\sr{P}(G/e) = (p)$. 

\begin{theorem}\label{thm:p-gens}
Let $G$ be a Abelian group, $\sr{P}$ be a prime ideal of $\burn_G$ over $p$. Then for all $K\leq H\leq G$ satisfying $p \mid |H:K|$, we have $H/K \in \sr{P}(G/H)$.
\end{theorem}
\begin{proof}
We will perform induction on the order of $H$. Note that by \cref{thm:simple_gens}. we have $p H/H \in \sr{P}(G/H)$. Consider any $K \leq H$ such that $|H:K| = pn$ for some $n\in \NN$. We wish to show that $H/K \in \sr{P}(G/H)$. Since $\sr{P}$ is prime, it is equivalent to show that $Q(\sr{P}, H/K, H/K)$ holds. Let $J$, $J'$ and $L$ such that $J,J'\leq L$ and $J,J'\leq H$ (as in \cref{def:Q} with $H' = H$). We wish to show that\[
(\N{J}{L}\circ \res{J}{H}(H/K))\cdot(\N{J'}{L} \circ \res{J'}{H}(H/K)) \in \sr{P}(G/L).
\]In the case that $J = J' = H$, we have $$H/K \times H/K = |H:K|H/K = (pH/H)\times n H/K\in\sr{P}(G/H).$$
The only case that remains to check is when one of $J,J'$ is a strict subgroup of $H$. Without loss of generality, take $J < H$. Our goal is now to show that $\res{J}{H}(H/K) \in \sr{P}(G/J)$, from which it follows that $\N{J}{L}\circ \res{J}{H}(H/K)\in \sr{P}(G/L)$.

The restriction $\res{H}{J}(H/K)$ is the set $H/K$ with action restricted to $J$. Thus the stabilizer of any element is $K\cap J$, so $$\res{J}{H}(H/K) = \frac{|H:K|}{|J:K\cap J|}J/(K\cap J).$$
If $p\mid |J:K \cap J|$, then the inductive hypothesis yields the desired conclusion. On the other hand, if $p\nmid |J:K \cap J|$ then we are done by \cref{thm:simple_gens} since $p \mid |H:K|$. Therefore $Q( \sr{P}, H/K, H/K)$ holds, implying $H/K \in \sr{P}(G/H)$ as desired.
\end{proof}

\begin{corollary}
Let $G$ be a finite Abelian group. Suppose that $\sr{P}$ is a prime ideal of $\burn_G$ over $p$. Then $\fr{p}_{G, p} \subseteq \sr{P}$.
\end{corollary}\begin{proof}
By \cref{rmk:pep gens}, we can describe $\fr{p}_{G, p}$ in terms of explicit generators. If $\sr{P}$ is a prime ideal over $p$, then \cref{thm:simple_gens} and \cref{thm:p-gens} imply that $\sr{P}$ contains these generators and hence contains $\fr{p}_{G,p}$.
\end{proof}

So far, we have been able to reduce the generators (1) and (2) in \cref{thm:first_gen} to the condition that $\sr{P}$ is prime and $p\cdot e/e \in \sr{P}(G/e)$. We continue this line of reasoning by attacking generators of the form (3). 

\begin{proposition}
Let $G$ be a cyclic group, $C \leq G$ and $\sr{I}$ be an ideal of $\burn_G$. Suppose that for all $L\leq G$ and all $L\cap C \leq K \leq L$,
$$L/K - |L:K|L/L \in \sr{I}(G/L) .$$
Then for all $L\leq G$, $J\leq C$ and each $M_J\leq L$ maximal amongst the subgroups of $L$ that intersect $C$ at $J$ and $J \leq H \leq M_J$ we have
$$L/H - |M_J:H|L/M_J \in \sr{I}(G/L).$$
\end{proposition}

\begin{proof}
Note that $M_J\cap C = J \leq H \leq M_J$, so by assumption $$M_J/H - |M_J:H|M_J/M_J \in \sr{I}(G/M_J)$$
which implies that 
 $$\tr{M_J}{L}(M_J/H - \abs{M_J:H}M_J/M_J) =  L/H - \abs{M_J:H}L/M_J \in \sr{I}(G/L).$$
\end{proof}

Our progress thus far can be summed up as follows:

\begin{corollary}\label{cor:gens}
Let $G$ be a cyclic group and suppose $\sr{P}$ is a prime ideal above $p$ and $C \leq G$. If for each $H\leq G$ and $H \cap C \leq K \leq H$ we have 
$$H/K - |H:K|H/H \in \sr{P}(G/H),$$
then 
$\fr{p}_{C, p} \subseteq \sr{P}$.
\end{corollary}

For $n$ dividing the order of (cyclic) $H\leq G$, we denote the unique transitive $H$-set of order $n$ in $A(H)$ by $t_n$. That is, $t_n = H/K$ where $K$ is the unique subgroup of $H$ with $\abs{H:K}=n$. We will also use the notation $\nu_p(n)\geq 0$ to denote the exponent of a prime $p$ in the prime factorization of $n$.

Multiplication in the Burnside ring is then given by
$$t_n \times t_m = \text{GCD}(n, m) t_{\text{LCM}(n, m)}.$$
For details on how the structure maps of $\burn_G$ interact with this notation, see \cite[\S 3]{calle/ginnett:2019}.
With this notation we can re-write the elements in the previous corollary as
$$t_{n}-n \in \sr{P}(G/H) \text{ where } n \mid |H:H\cap C|.$$

However, whenever $m$ and $n$ are relatively prime,
$$(t_n-n)\times (t_m-m) = t_{nm} - mt_n - n t_m + nm = (t_{nm} - nm) - n(t_m-m) - m(t_n-n),$$
so the elements in \cref{cor:gens} are generated by the smaller set
$$t_{p^k}-p^k \in \sr{P}(G/H) \text{ where } p^k\mid |H:C\cap H|.$$

However, via \cite[Lemma 3.6]{calle/ginnett:2019}, we can simplify these generators even more. In fact, for $K\leq H \leq G$, if $t_p-p\in\sr{I}(G/K)$, then $t_{p^{i+1}} - p^{i+1}\in\sr{I}(G/H)$ for all $0\leq i\leq \nu_p(\abs{H:K})$.
This knowledge combined with our previous discussion yields the following theorem:

%\begin{theorem}
%\label{thm:tamb_gens}
%Suppose $G$ is a cyclic group and $C\leq G$, and consider a prime ideal $\sr{P}\subseteq\burn_G$ over $q$. Suppose that, for each prime $p$ dividing the order of $G$ with $\text{Syl}_p(C) \neq \text{Syl}_p(G)$, the element$t_p-p$ is in $\sr{P}(G/\text{Syl}_p(C)^+)$,where $\text{Syl}_p(C)^+$ denotes the subgroup of $G$ of order $p|\text{Syl}_p(C)|$. Then $\fr{p}_{C, q} \subseteq \sr{P}$.

%In particular, if $\sr{P}|_{\text{Syl}_p(G)} = \fr{p}_{\text{Syl}_p(C),q}$ for each prime $p$ dividing the order of $G$, then 
%$\fr{p}_{C, q} \subseteq \sr{P}$.
%\end{theorem}

\begin{theorem}
\label{thm:tamb_gens}
Suppose $G=C_N$ is a cyclic group and $C_k\leq C_N$, and consider a prime ideal $\sr{P}\subseteq\burn_{C_N}$ over $q$. Suppose that for each prime $p\mid N/k$ that
$t_p-p$ is in $\sr{P}(G/C_{pk/p^{\nu(k)}})$. Then 
$\fr{p}_{C, q} \subseteq \sr{P}$.

In particular, if $\sr{P}|_{C_{N/p^{\nu(N)}}} = \fr{p}_{C_{k/p^{\nu(k)}},q}$ for each prime $p\mid N$, then 
$\fr{p}_{C_k, q} \subseteq \sr{P}$.
\end{theorem}

\section{Calculating the Spectrum for a finite Cyclic Group}\label{sec:pf}

We now turn to proving that our collection of ideals is the complete spectrum of the Burnside Tambara functor of a cyclic group, building upon previous literature which calculated the spectrum in the case of cyclic $p$-groups. 

\begin{theorem}[{Nakaoka \cite[Theorem 6.12]{nakaoka:2014}}]
\label{thm:Nak_main}
The Tambara spectrum of a cyclic $p$-group is
$$\Spec(\burn_{C_{p^n}}) = \{\fr{p}_{C, q} \mid C \leq C_{p^n}, q \text{ prime or zero}\}.$$
\end{theorem}

\begin{proof}
Nakaoka's description of these ideals is rather different from our $\fr{p}_{C,p}$ definition, and we will not detail his methods and notation here, but rather point the reader to \cite[Proposition 3.4, Definition 3.5]{nakaoka:2014}. Essentially, Nakaoka uses the Tambara structure maps to find the ``largest ideals over'' and the ``smallest ideals over'' some ideal $\sr{I} \subseteq \burn_{C_{p^n}}$, denoted $\mathcal{L}\sr{I}$ and $\mathcal{S}\sr{I}$, respectively. To make the comparison between our ideals and Nakaoka's explicit, we take $C= C_{p^m}$ where $m \leq n$. Then, in Nakaoka's notation,
\begin{equation*}
  \fr{p}_{C, q} = \begin{cases} 
       \mathcal{L}^n(p) & q = p,\\
        \mathcal{L}^{n-m}(0) & q = 0,\\
        \mathcal{L}^{n-m}\mathcal{S}^m(q) & \text{otherwise.}
    \end{cases}
\end{equation*}
The case for $q=p$ is trivial, as there is only one prime ideal $\fr{p}_{C,p}$ satisfying $\fr{p}_{C,p}(C_{p^n}/e) = (p)$. To see the comparison for the other cases, note that by \cref{lem:phi_res} the following conditions are equivalent:\begin{enumerate}
    \item[(i)] for all $H\leq C$, $\varphi^H_{C_{p^n}}(X) \equiv_q 0$,
    \item[(ii)] for each $H\leq C$, $q$ divides the coefficient in front of $H/H$ in $\res{H}{C_{p^n}}(X)$.
\end{enumerate} Then, due to the computational rules for restrictions, this is equivalent to the condition that every coefficient in $\res{C_{p^m}}{C_{p^n}}(X)$ is a multiple of $q$ for all $m \leq n$.

The ideal $\mathcal{L}^{n-m}\mathcal{S}^m(q)$ is defined by setting $\mathcal{L}^{n-m}\mathcal{S}^m(q)(C_{p^n}/H) = (q)$ for each $H \leq C_p^m$ and then defining the rest of the levels as the inverse restriction of the level below it. The ideal $\mathcal{L}^{n-m}(0)$ is defined similarly replacing $(q)$ with $(0)$. 
\end{proof}

We will need the following lemma, which allows us to build up prime ideals of $\burn_G$ from the prime ideals of $\burn_H$ for $H\leq G$.

\begin{lemma}[Restriction to Subgroups] \label{lem:res to subgp}
Let $G$ be a finite Abelian group and  $I \leq G$. Then if $\sr{P}$ is a prime ideal of $\burn_G$, the restriction $\sr{P}|_{I}$ of $\sr{P}$ onto $\burn_I$ is a prime ideal of $\burn_I$. 
\end{lemma}

\begin{proof}
Let $H, H' \leq I$, $a\in \burn_I(I/H) =\burn_G(G/H)$ and $b\in \burn_I(I/H')=\burn_G(G/H')$. The claim will follow from showing that if $Q(\sr{P}|_I,a, b)$ holds in $\burn_I$ then $Q(\sr{P}, a, b)$ holds in $\burn_G$. Assume it holds in $\burn_I$ and let $K, K', L \leq G$ be as in \cref{def:Q}. Since $G$ is Abelian, we may safely ignore the conjugation maps. Then, since $K, K' \leq I$, we have that $KK' \leq I$. Moreover, $KK' \leq L$, implying that
\begin{align*}
    (\N{K}{L}&\circ \res{K}{H}(a))\cdot(\N{K'}{L}\circ \res{K'}{H'}(b)) \\
    &= \N{KK'}{L}((\N{K}{KK'}\circ \res{K}{H}(a))\cdot(\N{K'}{KK'} \circ \res{K'}{H'}(b)))
\end{align*}
By assuming that $Q(\sr{P}|_I, a, b)$ holds, we have $(\N{K}{KK'}\circ \res{K}{H}(a))\cdot(\N{K'}{KK'} \circ \res{K'}{H'}(b)) \in \sr{P}|_I(I/KK') = \sr{P}(G/KK')$ implying that the desired product is in $\sr{P}(G/L)$, thus completing the proof. 
\end{proof}

Nakaoka's previous work means that we only need to establish the result when the order of $G$ is a composite. Before we do so, we recall the definition of the M\"obius function of a poset, which can be used to state a number of combinatorial properties (cf. \cite{godsil2018}). 

\begin{definition}
The M\"obius function of a finite poset is defined by \begin{equation*}
\mu(a, b) = \begin{cases}
0 & a \not\leq b, \\
1 & a = b, \\
- \sum_{x < b} \mu(a, x) & \text{otherwise.}
\end{cases}
\end{equation*}
The M\"obius function of a finite group $G$ is defined as the M\"obius function of the poset of subgroups of $G$ partially ordered by inclusion. If $G$ is a cyclic group and $H\leq K\leq G$, then
$$\mu(H, K) = \mu(|K:H|),$$
where the $\mu$ on the right side of the equality is the standard M\"obius function from number theory.
\end{definition}

\begin{theorem}
\label{thm:main_thm}
For a finite cyclic group $G$, $$\Spec(\burn_G) = \{\fr{p}_{C, p} \mid C \leq G, ~p \text{ prime or zero }\}.$$ 
\end{theorem}
\begin{proof}
The proof will proceed by induction on $n= |G|$. By Nakaoka's theorem, the result holds for all cyclic $p$-groups and so we may assume that $n$ is a composite with prime factorization $n=p_1^{\nu_1} \cdots p_m^{\nu_m}$. Let $\sr{P}$ be a prime ideal of $\burn_G$ over $q$. Assuming the result for all subgroups of order smaller than $n$, we will prove that $\sr{P}=\fr{p}_{C,q}$, where $C$ is constructed as follows:

By the restriction to subgroups lemma (\cref{lem:res to subgp}), for all $i = 1, \cdots, m$, we have $\sr{P}|_%{\text{Syl}_{p_i}(G)}
{C_{p_i^{\nu_i}}} = \fr{p}_{C_i, q}$ for some $C_i \leq%\text{Syl}_{p_i}(G)
C_{p_i^{\nu_i}}$. Define $C$ as the product
\begin{align*}
    C &= \prod_{i=1}^m C_i
\end{align*}
so that $C\cap %\text{Syl}_{p_i}(G) 
C_{p_i^{\nu_i}}= C_i$ for each $p_i \mid n$.

To show $\fr{p}_{C, q} \subseteq \sr{P}$, note that for all $p$-groups $H\leq %\text{Syl}_{p_i}(G)
C_{p_i^{\nu_i}}$, 
we have $\sr{P}(G/H) = \fr{p}_{C, q}(G/H) =  \fr{p}_{C_i, q}(%\text{Syl}_{p_i}(G)
C_{p_i^{\nu_i}}/H)$. It follows from \cref{thm:tamb_gens} that $\fr{p}_{C, q} \subseteq \sr{P}$. 

To show the reverse inclusion, by induction on $n$ we may assume that for all $H < G$, $\sr{P}(G/H) = \fr{p}_{H\cap C, q}(G/H) = \fr{p}_{C, q}(G/H)$. 
Therefore it suffices to show that the ideals are identical on $G/G$, so we take $X \in \sr{P}(G/G)$. For all $H \leq C$ with $H < G$, we have $\varphi_G^H(X) = \varphi_H^H(\res{H}{G}(X))$. Since $\res{H}{G}(X) \in \sr{P}(G/H) = \fr{p}_{C, q}(G/H)$ it follows that $\varphi_G^H(X) \in (q)$. If $C < G$, then $\varphi_G^H(X) \in (q)$ for all $H \leq C$; this is precisely the condition that $X \in \fr{p}_{C, q}(G/G)$. This shows that $\sr{P}(G/G) \subseteq \fr{p}_{C, q}$, completing the proof for $C<G$. 

In the case that $C = G$, all that remains to be shown is that $\varphi_G^G(X) \in (q)$.
Let $$X = \sum_{I\leq G}m_IG/I,$$
so then $$\varphi_G^G(X) = m_G.$$
We wish to show $m_G \in (q)$. Let $p \mid n$ be prime, so that there exists an $H\leq G$ with $|G:H| = p$. By the last paragraph we have
$$\varphi_G^H(X) = \sum_{H\geq I \geq G}m_I|G:I| = m_G + pm_H \in (q)$$
If $q \mid n$ we obtain the desired result by taking $p = q$. 

Now suppose that $q \nmid n$. Then for all $H\leq G$ with $|G:H| = p_i$ for $i = 1, \cdots m$ we have shown that 
$$ |G:H|m_H \equiv_q -m_G.$$ We claim that for all $H\leq G$, $|G:H|m_H \equiv_q \mu(H, G)m_G$. We have already shown the claim to hold for all $H$ with prime index; now let $H\leq G$ and assume the claim holds for all $H < K \leq G$. Then 
\begin{align*}
0 &\equiv_q \varphi_G^H(X) \\
  &\equiv_q \sum_{H \leq I \leq G} |G:I|m_I \\
  &\equiv_q |G:H|m_H + m_G\sum_{H < I \leq G}\mu(I, G) \\
  &\equiv_q |G:H|m_H - \mu(H, G)m_G,
\end{align*}
which shows that the claim holds in this case as well. Therefore by induction the claim holds for all subgroups of $G$. 

Let $K<G$ be the minimal subgroup of $G$ whose index is square free. Then since $\mu(H, G) = \mu(|G:H|)$ we have that $\mu(I, G) = \pm 1$ for all $I\geq K$ and $\mu(I, G) = 0$ for all $I \not\geq K$. Thus
$$m_G \equiv_q \frac{|G:K|m_K}{\mu(K, G)},$$
which, combined with the relation $|G:H|m_H \equiv_q \mu(H, G)m_G$, gives
\[
\frac{\abs{G:H}m_H}{\mu(H,G)} \equiv_q m_G \equiv_q \frac{\abs{G:K}m_K}{\mu(K,G)}
\] and hence
\[
m_H\equiv_q \frac{|G:K|m_K}{\mu(K, G)}\cdot\frac{\mu(H,G)}{\abs{G:H}} = \frac{\abs{H:K}\mu(H,G)}{\mu(K,G)}m_K.
\]
Therefore we can write
$$X = qY + m_K\sum_{K \leq H \leq G}\frac{|H:K|\mu(H, G)}{\mu(K, G)}G/H$$ for some some $Y\in \underline{A}_G(G/G)$.
Thus
\begin{align*}
0 &\equiv_q \varphi_G^K(X)\\ 
           &\equiv_q \varphi_G^K(X - qY) \\
           &= m_K\sum_{K \leq H \leq G}\frac{|H:K|\mu(H, G)}{\mu(K, G)}|G:H| \\
           &= m_K \frac{|G:K|}{\mu(K, G)}\sum_{K \leq H \leq G}\mu(H, G) \\
           &= m_K \frac{|G:K|}{\mu(K, G)}\mu(K, G) \\
           &= m_K |G:K|.
\end{align*}
Since we assumed $q$ to be relatively prime to $\abs{G}$, we therefore have that $m_K \in (q)$, implying (by our earlier relation) that $m_G \in (q)$, as desired.
\end{proof}

In the case of a cyclic group $C_n$, we can expand upon \cref{thm:containment} and explicate the relationship between the $\fr{p}_{C,p}$. A simplifying observation is that
\[
O^p(C_n) = C_{n/p^{\nu(n)}},
\]
where $p^{\nu(n)}$ is the highest power of $p$ that divides $n$. 

\begin{corollary}\label{rmk:cyclic containment} Let $G=C_n$, $C_i, C_j\leq C_n$, and $p,q$ prime. Then
\begin{enumerate}
    \item[(i)] $\fr{p}_{C_i, 0}\subseteq \fr{p}_{C_j, 0}$ if and only if $j\mid i$,
    \item[(ii)] $\fr{p}_{C_i, 0}\subset \fr{p}_{C_i, p}$ and $\fr{p}_{C_i, p}\not\subseteq\fr{p}_{C_j, 0}$,
    \item[(iii)] $\fr{p}_{C_i, p}\subseteq \fr{p}_{C_j, q}$ if and only if $p=q$ and $C_{i/p^{\nu(i)}}\leq C_{j/p^{\nu(j)}}$.
\end{enumerate}
\end{corollary}

These rules imply that the containment structure of ideals over $p$ is dual to the subgroup structure of $C_{n/p^\nu}$. The only connections between different ``levels'' are those given by (ii), namely, those between $0$ and a prime. These observations are highlighted in \cref{fig:incl Cn}.

\begin{figure}[h!]
    \centering
\begin{tikzpicture}[node distance=1.5cm,line width=0.25pt]
\title{Title}
\node(0) at (0,0)  {$\fr{p}_{e,0}$};
\node(Ci)       [above left  =0.2cm of 0] {$\fr{p}_{C_p,0}$};
\node(Cj)       [above right  =0.2cm of 0] {$\fr{p}_{C_q,0}$};
\node(Cij)       [above =1cm of 0] {$\fr{p}_{C_{pq},0}$};

\node(OpCn)       [above =1cm of Cij] {$\fr{p}_{O^p(C_n),0}$};
\node (vdotsu)       [above =0cm of OpCn] {$\vdots$};

\node(vdotsl)       [above =1.5cm of Ci] {$\vdots$};
\node(vdotsr)       [above =1.5cm of Cj] {$\vdots$};
\node(vdots)       [above =0.5cm of Cij] {$\vdots$};

\node(Cn)       [above =2.5cm of vdots] {$\fr{p}_{C_n,0}$};

\draw[<-] (0) -- (Ci);
\draw[<-] (0) -- (Cj);
\draw[<-] (0) -- (Cij);
\draw[<-] (Cj) -- (Cij);
\draw[<-] (Ci) -- (Cij);

\draw (vdotsl) -- (Cn);
\draw (vdotsr) -- (Cn);
\draw (vdotsu) -- (Cn);

\draw[->] (vdotsl) -- (Ci);
\draw[->] (vdotsr) -- (Cj);
\draw[->] (vdots) -- (Cij);

\node(p) at (4,0.75) {$\fr{p}_{e, p}$};
\node(Cip) [above left =0.2cm of p] {$\fr{p}_{C_p,p}$};
\node(Cjp) [above right =0.2cm of p] {$\fr{p}_{C_q,p}$};
\node(Cijp)       [above =1cm of p] {$\fr{p}_{C_{pq},p}$};
\node(OpCnp)       [above =1.75cm of Cijp] {$\fr{p}_{O^p(C_n),p}$};

\draw[gray, ->] (0) -- (p);
\draw[gray, ->] (Cj) -- (Cjp);
\draw[gray, ->] (Ci) -- (Cip);
\draw[gray, ->] (Cij) -- (Cijp);
\draw[gray, ->] (Cn) -- (OpCnp);
\draw[gray, ->] (OpCn) -- (OpCnp);

\draw[<-] (p) -- (Cjp);
\draw[<->] (p) -- (Cip);
\draw[<-] (p) -- (Cijp);
\draw[<-] (Cip) -- (Cijp);
\draw[<->] (Cijp) -- (Cjp);

\node(vdotslp)       [above =0.5cm of Cip] {$\vdots$};
\node(vdotsrp)       [above =0.5cm of Cjp] {$\vdots$};
\node(vdotsp)       [above =0.5cm of Cijp] {$\vdots$};

\draw (vdotslp) -- (OpCnp);
\draw (vdotsrp) -- (OpCnp);
\draw (vdotsp) -- (OpCnp);
\draw[->] (vdotslp) -- (Cip);
\draw[->] (vdotsrp) -- (Cjp);
\draw[->] (vdotsp) -- (Cijp);

\node(p') at (-4,0.75) {$\fr{p}_{e, p'}$};
\node(Cip') [above left =0.2cm of p'] {$\fr{p}_{C_p,p'}$};
\node(Cjp') [above right =0.2cm of p'] {$\fr{p}_{C_q,p'}$};
\node(Cijp')       [above =1cm of p'] {$\fr{p}_{C_{pq},p'}$};
\node(Cnp') [above=3.75cm of p'] {$\fr{p}_{C_n,p'}$};

\draw[gray, ->] (0) -- (p');
\draw[gray, ->] (Cn) -- (Cnp');
\draw[gray, ->] (Ci) -- (Cip');
\draw[gray, ->] (Cj) -- (Cjp');
\draw[gray, ->] (Cij) -- (Cijp');

\node(vdotslp')       [above =0.5cm of Cip'] {$\vdots$};
\node(vdotsrp')       [above =0.5cm of Cjp'] {$\vdots$};
\node(vdotsp')       [above =0.75cm of Cijp'] {$\vdots$};

\draw[<-] (p') -- (Cip');
\draw[<-] (p') -- (Cjp');
\draw[<-] (p') -- (Cijp');
\draw[<-] (Cjp') -- (Cijp');
\draw[<-] (Cip') -- (Cijp');

\draw (vdotslp') -- (Cnp');
\draw (vdotsrp') -- (Cnp');
\draw (vdotsp') -- (Cnp');

\draw[->] (vdotslp') -- (Cip');
\draw[->] (vdotsrp') -- (Cjp');
\draw[->] (vdotsp') -- (Cijp');
\draw[gray, ->] (OpCn) -- (vdotsp');

\node(ldots)       [left =5cm of OpCn] {$\dots$};
\node(rdots)       [right =5cm of OpCn] {$\dots$};

\end{tikzpicture}
    \caption{\textbf{Inclusion structure of prime ideals of $\burn_{C_n}$.} There is an arrow from $\fr{p}_{C_i,p}$ to $\fr{p}_{C_j,q}$ if $\fr{p}_{C_i,p}\subseteq\fr{p}_{C_j,q}$. At the $0$ and $p'$ levels (for $p'\nmid n$), the inclusion structure is dual to the subgroup lattice of $C_n$. If $p\mid n$, then some of the inclusions become equalities, precisely when $O^p(C_i)=O^p(C_j)$.}
    \label{fig:incl Cn}
\end{figure}
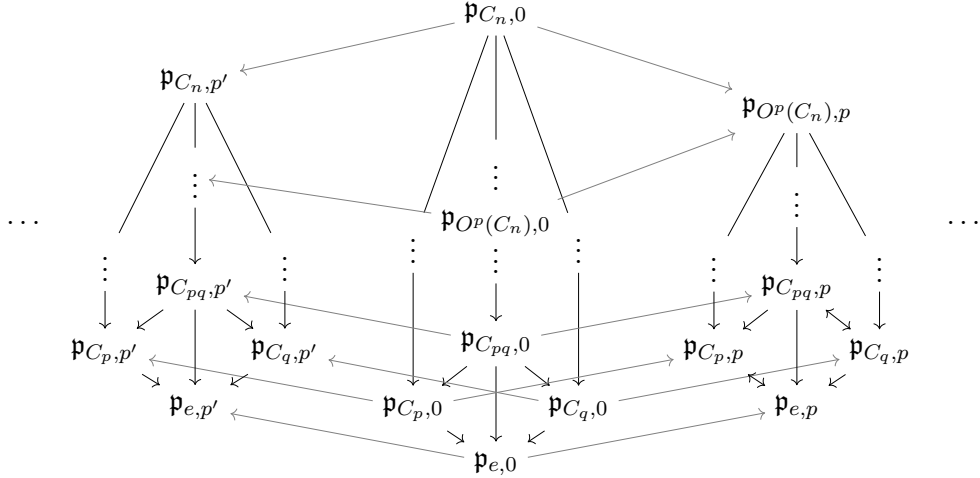

%--------------------------%--------------------------%--------------------------%--------------------------%--------------------------%--------------------------

\section{Directions for Further Investigation}

We strongly suspect that \cref{thm:main_thm} extends to a larger class of groups.
By \cref{thm:PCp prime}, we know that the family of ideals we have derived will be contained in the prime spectrum of the Burnside Tambara functor for an arbitrary finite group $G$.
 As we saw in \cref{thm:tamb_gens}, restricting to cyclic groups simplifies the structure of the ideals. However, this result only relies on the fact that every quotient of a cyclic group is a CS group, and there are other classes of groups that satisfy this property (most notably elementary Abelian groups and any Abelian group whose $p$-groups are all cyclic or elementary Abelian). One way to pursue a direct expansion of our results would be to give a version of \cref{thm:Nak_main} for elementary Abelian groups, along with generalizing a few places we used properties of cyclic groups in the proof of \cref{thm:main_thm} and \cref{thm:tamb_gens}. Alternatively, one could try to lift Dress's or Lewis's methods to Tambara functors, although the additional Tambara structure will undoubtedly introduce complications.

\bibliographystyle{abbrv}
\bibliography{references}
\end{document}